\documentclass{amsart}
\usepackage[utf8]{inputenc}
\usepackage{amsthm}

\usepackage{colonequals}
\usepackage[linesnumbered,ruled,vlined]{algorithm2e}
\usepackage{hyperref}
\newcommand{\Z}{\mathbb{Z}}
\newcommand{\Q}{\mathbb{Q}}

\let\C\relax
\newcommand{\C}{\mathbb{C}}
\let\P\relax
\newcommand{\P}{\mathbb{P}}

\newcommand{\Zp}{{\Z}_p}

\newcommand{\calP}{\mathcal{P}}
\newcommand{\Hp}{{\scriptsize H_p\left(\begin{matrix}\alpha\\\beta\end{matrix}\Big\vert z\right)}}
\newcommand{\Hpfull}{H_p\left(\begin{matrix}\alpha\\\beta\end{matrix}\Big\vert z\right)}
\newcommand{\Hvalue}[4]{{\scriptsize H_{#1}\left(\begin{matrix}#2\\#3\end{matrix}\Big\vert#4\right)}}
\newcommand{\Hvaluefull}[4]{H_{#1}\left(\begin{matrix}#2\\#3\end{matrix}\Big\vert#4\right)}

\numberwithin{equation}{section}

\theoremstyle{plain}

\newtheorem{theorem}[equation]{Theorem}

\newtheorem{proposition}[equation]{Proposition}
\newtheorem{lemma}[equation]{Lemma}

\theoremstyle{definition}
\newtheorem{definition}[equation]{Definition}
\newtheorem{notation}[equation]{Notation}
\newtheorem{example}[equation]{Example}

\theoremstyle{remark}
\newtheorem{remark}[equation]{Remark}

\DeclareMathOperator{\disc}{Disc}
\DeclareMathOperator{\frob}{Frob}
\DeclareMathOperator{\tr}{Tr}

\DeclareMathOperator{\id}{id}

\usepackage[color=white, linecolor=red, bordercolor=red]{todonotes}

\title{Hypergeometric $L$-functions in average polynomial time}
\author{Edgar Costa, Kiran S. Kedlaya, and David Roe}
\date{June 2020}

\thanks{Costa and Roe were supported by the Simons Collaboration on Arithmetic Geometry, Number Theory, and Computation via Simons Foundation grant 550033.
Kedlaya was supported by NSF (DMS-1802161) and UCSD (Warschawski Professorship).}

\begin{document}

\begin{abstract}
We describe an algorithm for computing, for all primes $p \leq X$, the mod-$p$ reduction of the trace of Frobenius at $p$ of a fixed hypergeometric motive in time quasilinear in $X$. This combines the Beukers--Cohen--Mellit trace formula with average polynomial time techniques of Harvey et al.
\end{abstract}

\maketitle

\section{Introduction}

In the past, computation of arithmetic $L$-functions has largely been limited to familiar classes of low-dimensional geometric objects, such as hyperelliptic curves or K3 surfaces~\cite{chk-18}.
Recently, it has emerged that families of motives whose associated (Picard-Fuchs) differential equation is a hypergeometric equation
also have $L$-functions which can be computed at large scale. Such motives provide accessible examples of arithmetic $L$-functions with diverse configurations of Hodge numbers, some of which arise in heretofore unanticipated applications.
For example, certain hypergeometric motives appear among families of Calabi--Yau threefolds, where they give rise to arithmetic manifestations of mirror symmetry (as in~\cite{doran-etal-18}).

Using finite hypergeometric sums in the manner of  Greene~\cite{greene-87}, Katz~\cite{katz-90}, and especially McCarthy~\cite{mccarthy-13}, 
an explicit formula for the $L$-function of a hypergeometric motive was given by Beukers--Cohen--Mellit~\cite{beukers-cohen-mellit-15}. It
was then modified by Cohen and Rodriguez Villegas, using the Gross-Koblitz formula~\cite{gross-koblitz-79} to replace classical Gauss sums with the Morita $p$-adic gamma function. That work is unpublished, but is documented in the manuscript~\cite{watkins-15}; the resulting formula appears in~\cite[\S 8]{cohen-15} and
\cite[\S 7.1]{fite-kedlaya-sutherland-16}; it is implemented in GP/PARI \cite{pari}, Magma \cite{magma}, and SageMath \cite{sage}; and it is being used to tabulate hypergeometric $L$-functions in the L-Functions and Modular Forms Database \cite{lmfdb}. (For an alternate approach using the $p$-adic Frobenius structure on a hypergeometric equation, see~\cite{kedlaya-19}.)

The purpose of this paper is to describe a preliminary adaptation of \emph{average polynomial time} techniques for computation of $L$-functions to the setting of hypergeometric motives. Such techniques, based on \emph{accumulating remainder trees}, were introduced by Costa--Gerbicz--Harvey~\cite{costa-gerbicz-harvey-14} for the problem of finding Wilson primes; adapted to computing $L$-functions by
Harvey~\cite{harvey-14, harvey-15}; and further elaborated (and made practical in particular cases)
by Harvey--Sutherland~\cite{harvey-sutherland-14,harvey-sutherland-16}
and Harvey--Massierer--Sutherland~\cite{harvey-massierer-sutherland}.

To simplify matters, we consider here only a limited form of the problem: given a hypergeometric motive over $\Q$ and a bound $X$, for each prime $p \leq X$, we compute the reduction modulo $p$ of the trace of Frobenius at $p$ in time quasilinear in $X$. This eliminates some technical issues that would arise when computing the mod-$p^e$ reduction for $e > 1$, such as the computation of multiplicative lifts
and evaluation of the Morita $p$-adic gamma function in average polynomial time. Modulo $p$, the trace formula at $p$ for a parameter value $t$ is a polynomial in $t$ of degree $O(p)$ whose coefficients are essentially ratios of Pochhammer symbols. Computing the Pochhammer symbols themselves in average polynomial time is a straightforward adaptation of the corresponding computation for factorials done in~\cite{costa-gerbicz-harvey-14}; this approach can then be modified to include the polynomial evaluation.  

At the end of the paper, we discuss the prospects of lifting our present restrictions of working modulo $p$ (rather than a higher power) and of computing only the trace of the $p$-power Frobenius (rather than a higher power).
Eliminating both restrictions would yield an average polynomial time algorithm for computing arbitrary hypergeometric $L$-series.
However, the restricted computation described here is already of significant value for hypergeometric motives of weight 1, for which the trace of the $p$-power Frobenius is determined uniquely by its reduction modulo $p$ (except when $p$ is very small). Since the formula for the trace of the $q$-power Frobenius involves a summation over $q-1$ terms, our method reduces the complexity of computing the first $X$ terms of the $L$-series from $X^2$ to $X^{3/2}$ (see Theorem~\ref{T:L-series use case}).

We end this introduction by asking (as in~\cite{kedlaya-19}) whether a similar trace formula exists for \emph{$A$-hypergeometric systems} in the sense of Gelfand--Kapranov--Zelevinsky~\cite{gelfand-kapranov-zelevinsky-08}.
Such a formula might unlock even more classes of previously inaccessible $L$-functions.

\section{Background}

\subsection{The \texorpdfstring{$p$-adic $\Gamma$}{p-adic Gamma} function}

For a detailed development of the following material, we recommend~\cite[\S~7.1]{robert-98} and~\cite[\S~6.2]{rodriquez-villegas-07}.

\begin{definition}
The (Morita) $p$-adic gamma function is the unique continuous function $\Gamma_p: \Z_p \to \Z_p^\times $ which satisfies
\begin{equation}
\label{equation:padicgamma-integers}
    \Gamma_p(n+1) = (-1)^{n+1} \prod_{\substack{i=1 \\ (i,p) = 1}} ^n i  = (-1)^{n+1} \frac{\Gamma(n+1)}{p^{\lfloor n/p \rfloor} \Gamma(\lfloor n/p \rfloor + 1)}
\end{equation}
for all $n \in \Z_{\geq 0}$.
For $p \geq 3$, it is Lipschitz continuous with $C = 1$, i.e.,
\begin{equation} \label{eq:Lipschitz continuity}
| \Gamma_p(x) - \Gamma_p(y)  |_p \leq | x - y|_{p}.
\end{equation}
There is also a functional equation analogous to the one for the complex $\Gamma$ function:
\begin{equation}\label{eq:gamma_feq}
    \Gamma_p (x + 1) = \omega(x)\Gamma_p (x),
    \qquad
    \omega(x) \colonequals
    \begin{cases} -x & \text{if } x \in \Z_p^\times \\
    -1 & \text{if } x \in p\Z_p.
    \end{cases}
\end{equation}
\end{definition}

\begin{remark}
It was originally observed by Dwork (writing pseudonymously in~\cite{boyarsky-80},
as corroborated in~\cite{katz-tate-99}; see \cite[\S 6.2]{rodriquez-villegas-07} for the formulation given here) that
$\Gamma_p$ admits an easily computable Mahler expansion on any mod-$p$ residue disc:
\begin{equation}
    \Gamma_{p}(-a + px) = \sum_{k \geq 0} p^k c_{a + kp} (x)_k,
\end{equation}
where $(x)_k \colonequals x(x+1) \cdots (x + k -1)$ is the usual Pochhammer symbol, and $c_n$ is defined by the recursion
\begin{equation}
n c_{n} = c_{n-1} + c_{n-p}, \qquad c_0 = 1, c_n = 0 \mbox{ for $n < 0$}.
\end{equation}
Thus, one may compute $\Gamma_p(x)$ modulo $p^f$ using $O(pf)$ ring operations.
\end{remark}

\subsection{Hypergeometric motives and their \texorpdfstring{$L$-functions}{L-functions}}

While the following discussion is needed to put our work in context, the reader is encouraged to skip ahead to \eqref{eq:Hq}, as the essential content of the paper is the computation of that formula.

\begin{definition}
A \emph{hypergeometric datum} is a pair of disjoint tuples
$\alpha=(\alpha_1,\ldots,\alpha_r)$ and $\beta=(\beta_1,\ldots,\beta_r)$
valued in $\Q \cap [0,1)$ which are \emph{Galois-stable} (or \emph{balanced}): any two reduced fractions with the same denominator occur with the same multiplicity.
\end{definition}

\begin{remark}\label{rem:cyclotomic data}
There are several equivalent ways to specify a hypergeometric datum.
One is to specify two tuples $A$ and $B$ for which the identity
\[
\prod_{j=1}^r \frac{x - e^{2\pi i \alpha_j}}{x - e^{2 \pi i \beta_j}}
= \frac{\prod_{a \in A} \Phi_a(x)}{\prod_{b \in B} \Phi_b(x)}
\]
holds in $\C(x)$, where $\Phi_n(x)$ denotes the $n$-th cyclotomic polynomial.
\end{remark}

\begin{definition}\label{def:zigzag function}
The \emph{zigzag function} $Z_{\alpha, \beta} : [0, 1] \to \Z$ associated to a hypergeometric datum $(\alpha, \beta)$ is defined by
\[
Z_{\alpha, \beta}(x) \colonequals \#\{j : \alpha_j \leq x\} - \#\{j : \beta_j \leq x\}.
\]
\end{definition}

\begin{notation} \label{notation:HGM}
We denote by $M^{\alpha,\beta}$ the putative (see Remark~\ref{R:motives})
hypergeometric family over $\P^1$ associated to the hypergeometric datum $(\alpha, \beta)$. Its expected properties are as follows.
\begin{itemize}
\item
It is a pure motive of degree $r$ with base field $\Q(t)$ and coefficient field $\Q$.
\item
Its Hodge realization is the one constructed by Fedorov in~\cite{fedorov-18}. This means that as per~\cite[Theorem~2]{fedorov-18},
its minimal motivic weight is
\begin{equation}\label{eq:mot_wt}
\begin{split}
w &= \max\{Z_{\alpha,\beta}(x): x \in [0,1]\} -
\min\{Z_{\alpha,\beta}(x): x \in [0,1]\} - 1 \\
&=  \max\{Z_{\alpha,\beta}(x): x \in \alpha\} -
\min\{Z_{\alpha,\beta}(x): x \in \beta\} - 1
\end{split}
\end{equation}
and a similar recipe (see~\cite[Conjecture~1.4]{corti-golyshev-11} or \cite[Theorem~1]{fedorov-18}) computes the Hodge numbers. Note that $rw$ is even \cite[1.2]{watkins-15}.
\item
Its $\ell$-adic \'etale realization is Katz's perverse sheaf \cite[Chapter~8]{katz-90}.
\item
For $z \in \Q \setminus \{0,1,\infty\}$,
let $M^{\alpha, \beta} _z$ denote the specialization of $M^{\alpha,\beta}$ at $t = z$.
Then the primes of bad reduction for $M^{\alpha,\beta}_z$ are those primes $p$ at which $z$ and $z-1$ are not both $p$-adic units (called \emph{tame} primes) and those primes $p$ at which the $\alpha_i$ and $\beta_i$ are not all integral (called \emph{wild} primes).
By the compatibility with Katz, the $L$-function associated to $M^{\alpha,\beta}_z$ is given by the Beukers--Cohen--Mellit trace formula \cite{beukers-cohen-mellit-15}.
\end{itemize}
\end{notation}

\begin{remark}
In order to avoid some case subdivisions in what follows, we assume hereafter that $0 \notin \alpha$.
This is relatively harmless because of the isomorphism
\begin{equation}
M^{\alpha,\beta}_z \cong M^{\beta,\alpha}_{1/z}.
\end{equation}
\end{remark}

\begin{example}
As per~\cite{ono-98},
$M^{(1/2, 1/2), (0, 0)}$ is the motive $H^1 (E, \Q)$, where
\begin{equation}
    E \colon y^2 = -x(x-1)(x-t).
\end{equation}
For other (putative) examples, see~\cite{barman-kalita-12} and~\cite{naskrecki-18}.
\end{example}

\begin{remark} \label{R:motives}
We use the qualifier ``putative'' in Notation~\ref{notation:HGM} for two reasons. One is to avoid any precision about motives;
while \cite{beukers-cohen-mellit-15} describes a specific variety whose $\ell$-adic cohomology includes Katz's perverse sheaf, lifting this containment to the motivic level would require a deeper dive into motivic categories (including a choice of which such category to consider). 

The other, more serious issue is that there is no existing reference that provides this missing precision on hypergeometric motives. 
The reader seeking to remedy this should start with \cite{andre} for a user's guide to motives.
\end{remark}

\subsubsection{Trace formulas}

We are particularly interested in computing
\begin{equation}\label{eqn:charpoly}
\det(1 - T \frob | M^{\alpha, \beta} _z),
\end{equation}
where $\frob$ is the Frobenius automorphism at a prime $p$ of good reduction for $M^{\alpha,\beta}_z$. (For concreteness, we may replace $M^{\alpha, \beta}_z$ with an \'etale realization.)
We ignore primes of bad reduction both because they are small enough to be handled individually and because a somewhat different recipe is required (see~\cite[\S~11]{watkins-15} for a partial description, noting that our $z$ is Watkins's $1/t$).

\begin{definition}
Let $\{x\} \colonequals x - \lfloor x \rfloor$ be the fractional part of $x$. For $q = p^f$, define
\begin{equation}
    \Gamma_q^*(x)\colonequals\prod_{v=0}^{f-1}\Gamma_p(\{p^vx\}),
\end{equation}
and then define a $p$-adic analogue of the Pochhammer symbol by setting
\begin{equation}\label{eq:poch_def}
(x)_{m}^* \colonequals \frac{\Gamma_q^*\left(x+\frac{m}{1-q}\right)}{\Gamma_q^*(x)}.
\end{equation}
Let $[z]$ be the multiplicative representative in $\Z_p$ of the residue class of $z$ (the unique $(p-1)$-st root of $1$ congruent to $z$ modulo $p$).
As in~\cite[\S~2]{watkins-15}, write
\begin{equation}\label{eq:Hq}
\Hvaluefull{q}{\alpha}{\beta}{z} \colonequals \frac{1}{1-q}\sum_{m=0}^{q-2}(-p)^{\eta_m(\alpha)-\eta_m(\beta)}q^{D + \xi_m(\beta)}\left(\prod_{j=1}^r \frac{(\alpha_j)_m^*}{(\beta_j)_m^*}\right)[z]^m,
\end{equation}
using the notations
\begin{align}
\eta_m(x_1,\ldots,x_r) &\colonequals\sum_{j=1}^r\sum_{v=0}^{f-1}\left\{p^v\left(x_j+\tfrac{m}{1-q}\right)\right\}-\left\{p^v x_j\right\}, \\
\xi_m(\beta) &\colonequals\#\{j:\beta_j=0\}-\#\left\{j:\beta_j+\tfrac{m}{1-q}=0\right\}, \label{eq:xi} \\
D &\colonequals \frac{w + 1 - \#\{j: \beta_j=0\}}{2}.
\end{align}
\end{definition}

By adapting~\cite[Theorem~1.3]{beukers-cohen-mellit-15} using the Gross--Koblitz formula as in~\cite[\S~2]{watkins-15} (and twisting by $q^D$ to minimize the weight), we deduce the following.
\begin{theorem}\label{thm:trace of Frobenius}
We have $\Hvaluefull{p^f}{\alpha}{\beta}{z} = \tr(\frob^f | M^{\alpha, \beta} _{z} ) \in \Z$.
\end{theorem}
From~\cite[\S~11]{watkins-15}, we also have a precise formula for the functional equation associated to $\det(1 - T\frob | M^{\alpha, \beta} _z)$.
\begin{theorem}
We have
\begin{equation}
\det(1 - q^{-w} T^{-1}  \frob | M^{\alpha, \beta} _z)
= \pm q^{-rw/2} T^{-r} \det(1 - T  \frob | M^{\alpha, \beta} _z)
\end{equation}
where $\pm$ denotes $+1$ if $w$ is even, and otherwise is given by
\[
\begin{cases}
\left( \frac{\Delta}{p} \right), \quad \Delta =
z(z-1) \prod_{a \in A} \disc(\Phi_a(x)) & \mbox{for } r \equiv 0 \pmod{2} \\
-\!\left( \frac{\Delta}{p} \right), \quad \Delta = (1-z) \prod_{b \in B} \disc(\Phi_b(x)) & \mbox{for } r \equiv 1 \pmod{2}.
\end{cases}
\]
Here $A,B,\Phi_a,\Phi_b$ are as in Remark~\ref{rem:cyclotomic data}
and $\left( \frac{\Delta}{p} \right)$ is the Kronecker symbol.
\end{theorem}

Using these two results, we can recover $\det(1 - T \frob | M^{\alpha, \beta} _z)$
from the values
$\Hvalue{p^f}{\alpha}{\beta}{z}$ for $f = 1, \dots, \lfloor \tfrac{r}{2} \rfloor$.

\subsubsection{Complexity considerations}

Computing $\Hvalue{p^f}{\alpha}{\beta}{z}$ via~\eqref{eq:Hq}
requires $O(f p^f)$ arithmetic operations,\footnote{The factor of $f$ comes from computing $\Gamma_p$. We do not incur a factor of $f$ from computing $\Gamma_q^*$ because the latter is invariant under $x \mapsto \{ px \}$, so we only need $O(q/f)$ evaluations of $\Gamma_q^*$.} due to the number of terms in the sum and product~\cite[\S 2.1.4]{watkins-15}. 
As these operations are in $\Z_p$, we must also pay attention to $p$-adic working precision; since $\Hvalue{p^f}{\alpha}{\beta}{z}$ is the sum of $r$ algebraic integers of complex norm $p^{wf/2}$, it is uniquely determined by its reduction modulo $p^e$ for $e > \tfrac{1}{2} wf + \log_p(2r)$.

For the use case of computing $L$-series, a different analysis applies.

\begin{theorem} \label{T:L-series use case}
Fix a hypergeometric datum $(\alpha, \beta)$. Given $\Hvalue{p}{\alpha}{\beta}{z}$ for all primes $p \leq X$, 
one can compute the first $X$ coefficients of the Dirichlet $L$-series associated to $M^{\alpha,\beta}_z$ in at most
$O(X^{3/2})$ arithmetic operations.
\end{theorem}
\begin{proof}
The first $X$ coefficients of the Dirichlet series are determined by the coefficients indexed by prime powers up to $X$,
and hence by the values $\Hvalue{q}{\alpha}{\beta}{z}$ for all prime powers $q \leq X$.
The number of such $q$ which are not prime is $O(X^{1/2}/\log X)$; for $q = p^f$, evaluating \eqref{eq:Hq} takes $O(f p^f) = O(X \log X)$ arithmetic operations.
\end{proof}

\section{Accumulating remainder trees}\label{sec:art}

The use of a \emph{remainder tree} to expedite modular reduction has its origins in the fast Fourier transform (FFT).
An early description was given by Borodin--Moenck~\cite{borodin-moenck-74}; for a modern treatment
with more historical references, see~\cite{bernstein-08}.

\emph{Accumulating remainder trees} were introduced in~\cite{costa-gerbicz-harvey-14}
in order to compute $(p-1)! \pmod{p^2}$ for many primes $p$.
We use the variant described in~\cite[\S 4]{harvey-sutherland-14}.

\begin{definition}\label{def:acc_rem_tree notation}
Suppose $\calP$ is a sequence $p_1, \ldots, p_{b-1}$ of pairwise coprime integers with $p_i \le X$, and $A_0, \ldots, A_{b-2}$
is a sequence of $2 \times 2$ integer matrices.
We may use an accumulating remainder tree to compute
\begin{equation}
C_n \colonequals A_0 \cdots A_{n-1} \bmod p_n
\end{equation}
for $1 \le n < b$ as follows.
For notational convenience we assume $b = 2^\ell$, set $A_{b-1} = 0$ and $p_0 = 1$.
Then as in~\cite[\S~4]{harvey-sutherland-14}, write
\begin{equation}
\begin{aligned}
m_{i,j} &\colonequals p_{j2^{\ell-i}} p_{j2^{\ell-i}+1} \cdots p_{(j+1)2^{\ell-i}-1}, \\
A_{i,j} &\colonequals A_{j2^{\ell-i}} A_{j2^{\ell-i}+1} \cdots A_{(j+1)2^{\ell-i}-1}, \\
C_{i,j} &\colonequals A_{i, 0} \cdots A_{i, j-1} \bmod m_{i,j}.
\end{aligned}
\end{equation}
\end{definition}
This leads us to Algorithm~\ref{alg:acc_rem_tree}.
\begin{algorithm}
\DontPrintSemicolon
\SetKwProg{Fn}{def}{\string:}{}
\SetKwInOut{Input}{Input}
\SetKwInOut{Output}{Output}
\SetKwFunction{RemTree}{RemTree}

\Fn{\RemTree{$\{A_{i}\}$, $\{p_i\}$}}{
\Input{$A_0, \ldots, A_{b-1}, p_0, \ldots, p_{b-1}$ as in Definition~\ref{def:acc_rem_tree notation}}
%\Output{$C_1, \ldots, C_{b-1}$}
\Output{$\{C_{i}\}$}

\For{$j \colonequals 0$ \KwTo $b-1$}{
    $m_{\ell, j} \colonequals p_j$ and $A_{\ell, j} \colonequals A_j$\;
}
\For{$i \colonequals \ell - 1$ \KwTo $0$}{
    \For{$j \colonequals 0$ \KwTo $2^i-1$}{
        $m_{i,j} \colonequals m_{i+1, 2j} m_{i+1, 2j+1}$ and $A_{i,j} \colonequals A_{i+1, 2j} A_{i+1, 2j+1}$\;
    }
}
$C_{0,0} \colonequals \id$\;
\For{$i \colonequals 1$ \KwTo $\ell$}{
    \For{$j \colonequals 0$ \KwTo $2^i-1$}{
        \If{$j$ even}{
            $C_{i,j} \colonequals C_{i-1,\lfloor j/2\rfloor} \bmod m_{i,j}$\;
        } \Else {
            $C_{i,j} \colonequals C_{i-1, \lfloor j/2\rfloor} A_{i,j-1} \bmod m_{i,j}$\;
        }
    }
}
\Return{${\{C_{\ell,j}\}}_{j=1, \dots, b-1}$}
%\Return{$\{C_{\ell,i}\}_{i = 1, \dots, b-1}$}
\caption{Accumulating Remainder Tree}\label{alg:acc_rem_tree}
}
\end{algorithm}

\begin{theorem}[{\cite[Thm. 4.1]{harvey-sutherland-14}}]\label{thm:acc_rem_tree analysis}
Let $B$ be an upper bound on the bit size of $\prod_{j=0}^{b-1} p_j$ and $H$ an upper bound
on the bit size of any $p_i$ or $A_i$.
The running time of Algorithm~\ref{alg:acc_rem_tree} is
\[
O((B + bH)\log(B + bH)\log(b))
\]
(using \cite{harvey-vanderhoeven-19} for the runtime of integer multiplication) and its space complexity is
\[
O((B + bH)\log(b)).
\]
\end{theorem}

\subsection{Accumulating remainder tree with spacing}\label{sec:synchronization}
In most applications (including this one),
there is not a one-to-one correspondence between the moduli $p_i$ and the multiplicands $A_i$.
Rather, we will be given:
\begin{itemize}
\item
a list of matrices $A_0,\dots,A_{b-1}$,
\item
a list of primes $p_1,\dots,p_c$, and
\item
a list of distinct cut points $b_1,\dots,b_c$
\end{itemize}
with the aim of computing $C_n \colonequals A_0 \cdots A_{b_n-1} \bmod p_n$
for $1 \le n < c$.
This reduces to Algorithm~\ref{alg:acc_rem_tree} by suitably grouping terms; see Algorithm~\ref{alg:acc_rem_tree_with_spacing}.
(One may also handle repeated cut points,
as long as the cut points up to $X$ occur at most $O(X)$ times.)

\begin{algorithm}
\DontPrintSemicolon
\SetKwProg{Fn}{def}{\string:}{}
\SetKwInOut{Input}{Input}
\SetKwInOut{Output}{Output}
\SetKwFunction{RemTree}{RemTree}
\SetKwFunction{RemTreeWithSpacing}{RemTreeWithSpacing}

\Fn{\RemTreeWithSpacing{{$\{A_{i}\}$, $\{p_i\}$}, $\{b_i\}$}}{
\Input{$A_0, \ldots, A_{b-1}, p_1, \ldots, p_{c}, b_1,\ldots,b_c$ as in Section~\ref{sec:synchronization}}
\Output{$C_1, \dots C_{c-1}$}

$\ell \colonequals \lceil \log_2(b) \rceil$\;

\For{$j \colonequals b$ \KwTo $2^\ell-1$}{
    $A_j \colonequals 0$\;
}

\For{$j \colonequals 0$ \KwTo $2^\ell-1$}{
    $p'_j \colonequals 1$\;
}

\For{$i \colonequals 1$ \KwTo $c$}{
  $p'_{b_i} \colonequals p_i$\;
}
$C'_i \colonequals \RemTree{$\{A_i\}, \{p'_i\}$}$\;
% RemTree{A_i, p'_i}
% Returns C'_n = v A_0 ... A_{n-1} mod p'_n
% in particular
% C'_{b_n} = v A_0 ... A_{b_n-1} mod p'_{b_n} = v A_0 ... A_{b_n-1} mod p_n
\Return{ $\{C'_{b_i}\}_{i=0,\dots,c-1}$ }

\caption{Accumulating Remainder Tree with Spacing}\label{alg:acc_rem_tree_with_spacing}
}
\end{algorithm}

\begin{remark} \label{rmk:remainder forest}
In practice, we will split our products to work around discontinuities of \eqref{eq:Hq}
(see Section~\ref{subsection:matrix product}). One gains some savings (particularly in space complexity)
by splitting a bit further, replacing remainder trees with \emph{remainder forests}
\cite[Theorem~4.2]{harvey-sutherland-14}; we omit the details here.
\end{remark}

\section{Nuts and bolts}

We record two technical lemmas used in the description of our algorithm.
For the rest of the paper, we make the simplifying assumption $q = p$ in Theorem \ref{thm:trace of Frobenius}.

\begin{lemma}\label{lem:shift}
Set $I_b \colonequals [0, 1] \cap \frac{1}{b}\Z$.
Suppose $\gamma \in I_b$ and $p$ is a prime not dividing $b$.
Let $m = \lfloor \gamma (p-1) \rfloor$.  Then there exist $\delta \in I_b$ and $\epsilon \in \{1, 2\}$ so that
\[
m + \epsilon \equiv \delta \pmod{p}.
\]
Moreover, $\delta$ and $\epsilon$ only depend on $b$, $\gamma$, and $p \pmod{b}$.
\end{lemma}
\begin{proof}
Write $\gamma = \frac{a}{b}$ and define an integer $r \in \{0,\dots,b-1\}$ by the condition that
\[
a(p-1) = mb + r.
\]
We then set
\[
\begin{cases}
\epsilon \colonequals 1, \delta \colonequals \frac{b - a - r}{b} & \mbox{if } a+r < b \\
\epsilon \colonequals 2, \delta \colonequals \frac{2b-a-r}{b} & \mbox{otherwise.}
\end{cases}
\]

Note that $b(\delta - \epsilon) = -(a + r) = mb - ap$ so $m + \epsilon \equiv \delta \pmod{p}$.
The fact that $\delta \in I_b$ follows from the bounds $0 \le a, r \le b$.
\end{proof}

%\begin{lemma}\label{lem:eta_zigzag}
%Suppose $q = p$ and $0 \le m < p - 1$.  Then $\eta_m(\alpha) - \eta_m(\beta) = Z_{\alpha, \beta}(m/(p-1))$.
%\end{lemma}
%\begin{proof}
%\end{proof}

\begin{lemma}\label{lem:zigzag_invariance}
Suppose $0 \le m < p - 1$ and either $\eta_m(\alpha) - \eta_m(\beta) \ne \eta_{m+1}(\alpha) - \eta_{m+1}(\beta)$ or $\xi_m(\beta) \ne \xi_{m+1}(\beta)$.  Then $ \lfloor \gamma (p-1) \rfloor \in \{m,m+1\}$ for some $\gamma \in \alpha \cup \beta$.
\end{lemma}
\begin{proof}
Since $q=p$, we have
\begin{equation}\label{eq:express eta}
\eta_m(\alpha) - \eta_m(\beta) = \sum_{j=1}^r \left(\left\{ \alpha_j - \tfrac{m}{p-1} \right\} - \{\alpha_j\} \right) - \sum_{j=1}^r \left( \left\{ \beta_j - \tfrac{m}{p-1} \right\} - \left\{\beta_j\right\} \right).
\end{equation}
For $x,y \in [0,1)$ we have
\begin{equation}\label{eq:fractional part difference}
\{x-y\} = \begin{cases} x-y & (x \geq y) \\
x-y+1 & (x < y).
\end{cases}
\end{equation}
Consequently, the only way for $\eta_m(\alpha) - \eta_m(\beta)$
to change values when $m$ goes to $m+1$ is
for there to exist $\gamma \in \alpha \cup \beta$ such that
\[
\gamma - \tfrac{m}{p-1} \geq 0, \qquad \gamma - \tfrac{m+1}{p-1} < 0.
\]
This occurs precisely when $m = \lfloor \gamma(p-1) \rfloor$.
Meanwhile, by \eqref{eq:xi}, $\xi_m(\beta) = \xi_{m+1}(\beta)$
unless $\beta_j = \frac{m}{p-1}$ or $\beta_j = \frac{m+1}{p-1} = 0$ for some $j$.
\end{proof}

\section{Computing trace functions of hypergeometric motives}

Throughout this section, fix $\alpha, \beta$ and $z$.
We now describe how to compute the trace $\Hp$ modulo $p$ in average polynomial time using equation~\eqref{eq:Hq}, which we duplicate here modulo $p$ for ease of reference:
\begin{equation}\label{eq:Hp mod p}
\Hpfull \equiv \sum_{m=0}^{p-2}(-p)^{\eta_m(\alpha)-\eta_m(\beta)} p^{D + \xi_m(\beta)}\left(\prod_{j=1}^r \frac{(\alpha_j)_m^*}{(\beta_j)_m^*}\right)z^m \pmod{p}.
\end{equation}

\subsection{Overview of the algorithm}

In order to apply Algorithm~\ref{alg:acc_rem_tree_with_spacing},
we would like to identify $2 \times 2$ integer matrices $B(m)$, such that we may extract $\Hvalue{p}{\alpha}{\beta}{z} \pmod{p}$ from $B(0) B(1) \cdots B(p-2)$.
In practice, we will consider shorter subproducts and choose $B(m)$ based on the residue of $p$ modulo a fixed integer (independent of $m$ and $p$);
we will then apply Algorithm~\ref{alg:acc_rem_tree_with_spacing} once for each subproduct and residue class.

As a first approximation, let us instead model the sum $\sum_{m=0}^{p-2} P_m$ where
\begin{equation}\label{eq:Pm_def}
P_m \colonequals z^m\prod_{j=1}^r \frac{(\alpha_j)_m^*}{(\beta_j)_m^*} \in \Zp^\times.
\end{equation}
If we can find $f(m), g(m) \in \Z[m]$ so that
\begin{equation}\label{eq:poch_recurrence}
P_{m+1} \equiv \frac{f(m)}{g(m)} P_m \pmod{p},
\end{equation}
we can then set
\begin{equation}
B(m) \colonequals \begin{pmatrix} g(m) & 0 \\ g(m) & f(m) \end{pmatrix} = g(m) \begin{pmatrix} 1 & 0 \\ 1 & f(m)/g(m) \end{pmatrix}
\end{equation}
and $\widetilde{B} = B(0) \dots B(p-2) \pmod{p}$, so that
\[
\widetilde{B} \equiv g(0) \cdots g(p-2) \begin{pmatrix} 1 & 0 \\
\sum_{m=0}^{p-2} P_m & P_{p-1}
\end{pmatrix} \pmod{p}
\]
and so $\sum_{m=0}^{p-2} P_m \equiv \widetilde{B}_{21}/\widetilde{B}_{11} \pmod{p}$.
That is, $\widetilde{B}_{11}$ tracks a common denominator, $\widetilde{B}_{22}$ tracks the product $P_m$,
and $\widetilde{B}_{12}$ computes the sum of the $P_m$.

There are two problems with the approach described above.
First, to correctly simulate ~\eqref{eq:Hp mod p} we must sum not $P_m$ but
\begin{equation}\label{eq:Pm_def with scale factor}
P'_m \colonequals(-p)^{\eta_m(\alpha) - \eta_m(\beta)} p^{D+\xi_m(\beta)} P_m,
\end{equation}
which we cannot directly handle by modifying $B(m)_{21}$ because the extra factor depends on both $p$ and $m$.
Second, while we can find polynomials $f$ and $g$ satisfying~\eqref{eq:poch_recurrence} for most values of $m$
using~\eqref{eq:poch_def} and the functional equation~\eqref{eq:gamma_feq}, there will be a few values of $m$ where $f(m)$ or $g(m)$ is a multiple of $p$.  We cannot filter these values out during the remainder tree because $p$ is not fixed.

The solution to both of these issues is to break up the range $[0, p-2]$ into intervals on which equation~\eqref{eq:poch_recurrence} holds and
the values $\eta_m(\alpha) - \eta_m(\beta)$ and $\xi_m(\beta)$ are constant.
The breaks between these intervals occur when $m = \lfloor \gamma (p-1) \rfloor$, where $\gamma \in \alpha \cup \beta$.
We thus use a separate accumulating remainder tree for each interval, yielding for each $p$ a fixed number of subproducts with isolated missing terms in between; we then compute separately for each $p$ to bridge the gaps.

A third issue is that while we can vary the endpoint in an accumulating remainder tree as a function of $p$
(as described in Section~\ref{sec:art}), it is more difficult to change the start point.
Our solution is to use Lemma~\ref{lem:shift} to find a rational number $\delta$ so that adding $\delta$ to each $\alpha_j$ and $\beta_j$ has the effect of
shifting the start point to $0$.

\subsection{Construction of the matrix product}
\label{subsection:matrix product}

We now construct the matrix product described above. We begin with the division of the interval $[0,p-1]$ and the division of primes into residue classes.  We assume that $q = p$ is good and not $2$.
\begin{definition}
Given a hypergeometric motive $M^{\alpha, \beta}_z$, let $0=\gamma_0 < \dots < \gamma_s = 1$ be the distinct elements in $\alpha \cup \beta \cup \{0, 1\}$.
Let $b$ be the least common denominator of $\alpha \cup \beta$ and fix $c \in (\Z/b \Z)^\times$.
Let $p$ be a prime congruent to $c$ modulo $b$ and not dividing the denominator of $z$.
Write $m_i$ for $\lfloor \gamma_i (p-1) \rfloor$.
\end{definition}

We next exhibit polynomials that we use to compute Pochhammer symbols and their partial sums on the interval $(\gamma_i, \gamma_{i+1})$.

\begin{definition}
Fix an interval $(\gamma_i, \gamma_{i+1})$, choose $\delta_i$ and $\epsilon_i$ associated to $\gamma_i$ as in Lemma~\ref{lem:shift}, and let
\begin{equation}
  \iota(x, y) \colonequals \begin{cases} 1 & x \le y \\ 0 & x > y. \end{cases}
\end{equation}
Define polynomials $f_{i, c}(k), g_{i, c}(k) \in \Z[k]$ as follows: set
\begin{equation}\label{eq:f and g}
  \begin{aligned}
    F_{i, c}(k) &\colonequals z \prod_{j=1}^r (\alpha_j +  \delta_i +  \iota(\alpha_j, \gamma_i) + k - \epsilon_i) \\
    G_{i, c}(k) &\colonequals \prod_{j=1}^r (\beta_j + \delta_i + \iota(\beta_j, \gamma_i)  + k- \epsilon_i),
  \end{aligned}
\end{equation}
let $d_{i, c}$ be the least common multiple of the denominators of $F_{i, c}$ and $G_{i, c}$, and set $f_{i, c}(k) \colonequals d_{i, c}F_{i, c}(k)$ and $g_{i, c}(k) \colonequals d_{i, c}G_{i, c}(k)$.
\end{definition}

\begin{lemma}\label{lem:top left entry}
Define $P_m$ as in~\eqref{eq:Pm_def}, and suppose $m_i < m < m_{i+1}$.  Then
\[
P_{m+1} \equiv \frac{f_{i, c}(k)}{g_{i, c}(k)} P_m \pmod{p},
\]
where $1 \leq k < m_{i+1} - m_i$ and $m = m_i + k$.
\end{lemma}

\begin{proof}
We first focus on a single Pochhammer symbol $(\alpha_j)_m^*$.
  First note that for
  $m_i < m \leq m_{i+1}$, by~\eqref{eq:fractional part difference} we have
  \begin{equation} \label{eq:effect of shift}
    \left\{\alpha_j + \tfrac{m}{1-p}\right\} =
    \alpha_j + \tfrac{m}{1-p} +
    \begin{cases}
       0 & m \leq \left\lfloor \alpha_j(p-1) \right\rfloor\\
       1 & m > \left\lfloor \alpha_j(p-1) \right\rfloor
     \end{cases} = \alpha_j + \tfrac{m}{1-p} + \iota(\alpha_j, \gamma_i).
  \end{equation}
  Combining \eqref{eq:effect of shift} with Lipschitz continuity~\eqref{eq:Lipschitz continuity} and the functional equation for $\Gamma_p$~\eqref{eq:gamma_feq} and Lemma~\ref{lem:shift}, for $m_i < m < m_{i+1}$ we obtain
\begin{equation}\label{use of functional equation}
  \begin{aligned}
    \Gamma_p\left(\left\{\alpha_j + \tfrac{m+1}{1-p} \right\}\right)
    & \equiv
    \Gamma_p\left(\alpha_j + m + 1 + \iota(\alpha_j, \gamma_i)\right) \\
    &  = -(\alpha_j + m + \iota(\alpha_j, \gamma_i))
    \Gamma_p(\alpha_j + m + \iota(\alpha_j, \gamma_i)) \\
    &
    \equiv -(\alpha_j + \delta_i + \iota(\alpha_j, \gamma_i)  + k - \epsilon_i)
    \Gamma_p\left(\left\{\alpha_j + \tfrac{m}{1-p} \right\}\right) \pmod{p}.
  \end{aligned}
\end{equation}
Taking the product over all the Pochhammer symbols, the minus sign cancels out, and we obtain
the equation~\eqref{eq:f and g}, as desired.
\end{proof}

We next account for the power of $p$ in the product, and assemble a matrix product that computes the sum between two breaks.
\begin{definition}\label{def:sigma i}
Let $\xi(\beta) = \#\{j : \beta_j = 0\}$ and
\begin{equation}\label{eq:sigmai}
\sigma_i \colonequals \begin{cases}
1 & Z_{\alpha, \beta}(\gamma_i) + \xi(\beta) + D= 0 \mbox{ and } Z_{\alpha, \beta}(\gamma_i) \equiv 0 \pmod{2} \\
-1 & Z_{\alpha, \beta}(\gamma_i) + \xi(\beta) + D= 0 \mbox{ and } Z_{\alpha, \beta}(\gamma_i) \equiv 1 \pmod{2} \\
0 & \mbox{otherwise.}
\end{cases}
\end{equation}
By Lemma~\ref{lem:zigzag_invariance}, $\sigma_i$ gives the value of $(-p)^{\eta_m(\alpha) - \eta_m(\beta)} p^{\xi_m(\beta) + D} \bmod{p}$ for all $m$ with $m_i < m < m_{i+1}$. Now set
\begin{equation}\label{eq:Aic}
A_{i, c}(k) \colonequals \left(\begin{matrix} g_{i, c}(k) & 0 \\ \sigma_i g_{i, c}(k) & f_{i, c}(k) \end{matrix} \right).
\end{equation}
Since $A_{i, c}(k)$ depends only on $c$ and not $p$, we can use an accumulating remainder tree for each $c$ to compute
\begin{equation}\label{eq:Sip}
S_i(p) \colonequals A_{i, c}(1) A_{i, c}(2) \cdots A_{i, c}(m_{i+1} - m_i  - 1) \pmod{p}.
\end{equation}
\end{definition}

\begin{lemma}\label{lem:single interval}
For $P'_m$ as defined in~\eqref{eq:Pm_def with scale factor},
\begin{equation}
S_i(p)_{11}^{-1} S_i(p) \equiv \begin{pmatrix}
1 & 0 \\
\sum_{m=m_i+1}^{m_{i+1}-1} P'_{m} / P_{m_i + 1} & P_{m_{i+1}} / P_{m_i + 1}
\end{pmatrix} \pmod{p}.
\end{equation}
\end{lemma}
\begin{proof}
By Lemma~\ref{lem:top left entry}, for $k=1,\dots,m_{i+1}-m_i-1$,
\[
\frac{(A_{i,c}(1) \cdots A_{i,c}(k))_{22}}{(A_{i,c}(1) \cdots A_{i,c}(k))_{11}}
\equiv \frac{P_{m_i+k+1}}{P_{m_i+1}} \pmod{p}
\]
and hence
\[
\frac{(A_{i,c}(1) \cdots A_{i,c}(k))_{21}}{(A_{i,c}(1) \cdots A_{i,c}(k))_{11}}
\equiv \sigma_i \sum_{l=1}^{k} \frac{P_{m_i+l}}{P_{m_i+1}}
\pmod{p}.
\]
Taking $k = m_{i+1}-m_i- 1$, then applying Lemma~\ref{lem:zigzag_invariance} to replace
$\sigma_i$ with $P'_m/P_m$,
yields the desired result.
\end{proof}

It remains to deal with the breaks.
Since the number of breaks is independent of $p$, we have the luxury of computing matrices $T_i(p)$ separately for each $p$ that move the Pochhammer symbols and partial sums past the break $\gamma_i$.
\begin{definition}\label{def:htau}
With $\omega$ defined as in~\eqref{eq:gamma_feq}, let
\begin{align}
h_i(\gamma, p) &\colonequals \begin{cases}
\omega(\gamma + m_i + 1) & \mbox{if } \gamma(p-1) < m_i \\
\omega(\gamma + m_i) & \mbox{if } \gamma(p-1) \ge m_i +1 \\
\omega(\gamma + m_i + 1)\omega(\gamma + m_i) & \mbox{otherwise}
\end{cases} \\
\tau_i &\colonequals \begin{cases}
0 & \gamma_i = 0 \\
1 & Z_{\alpha, \beta}(\gamma_{i-1}) + \xi_{m_i}(\beta)  + D= 0 \mbox{ and } Z_{\alpha, \beta}(\gamma_{i-1}) \equiv 0 \pmod{2} \\
-1 & Z_{\alpha, \beta}(\gamma_{i-1}) + \xi_{m_i}(\beta) + D = 0 \mbox{ and } Z_{\alpha, \beta}(\gamma_{i-1}) \equiv 1 \pmod{2} \\
0 & \mbox{otherwise}
\end{cases}
\end{align}
and then set
\begin{align}\label{eq:Tip}
T_i(p) &\colonequals \begin{pmatrix} 1 & 0 \\ \tau_i & z \prod_{j=1}^r \frac{h_i(\alpha_j, p)}{h_i(\beta_j, p)} \end{pmatrix} \\
S(p) &\colonequals \prod_{i = 0}^{s-1} T_i(p) S_i(p).
\end{align}
Note that modulo $p$, $T_i(p)$ is congruent to a matrix that depends on $p$ only via $c$.
\end{definition}

\begin{lemma}\label{lem:single interval2}
For suitable choices of scalars, we have
\begin{align*}
\prod_{j=0}^{i-1} T_j(p) S_j(p) &\equiv \mbox{(scalar)} \begin{pmatrix} 1 & 0 \\
\sum_{m=0}^{m_i-1} P'_m & P_{m_i} \end{pmatrix} \pmod{p} \\
\left( \prod_{j=0}^{i-1} T_j(p) S_j(p) \right)  T_i(p)
&\equiv \mbox{(scalar)} \begin{pmatrix} 1 & 0 \\
\sum_{m=0}^{m_i} P'_m & P_{m_i+1} \end{pmatrix} \pmod{p}.
\end{align*}
\end{lemma}
\begin{proof}
This follows by induction on $i$ using Lemma~\ref{lem:single interval}.
\end{proof}

Summing up, we obtain the following.
\begin{proposition}\label{prop:mod p calculation}
For $p \equiv c \pmod{b}$ not dividing the denominator of $z$,
\[
\Hpfull \equiv S(p)_{21} / S(p)_{11} \pmod{p}.
\]
\end{proposition}
\begin{proof}
This follows from~\eqref{eq:Hp mod p} and the case $i=s$ of Lemma~\ref{lem:single interval2}.
\end{proof}

\subsection{Algorithm and runtime}

We summarize with Algorithm \ref{alg:traces}.

\begin{algorithm}[h]
\DontPrintSemicolon
\SetKwProg{Fn}{def}{\string:}{}
\SetKwInOut{Input}{Input}
\SetKwInOut{Output}{Output}
\SetKw{And}{and}
\SetKwData{Gammas}{gamma}
\SetKwData{Result}{result}
\SetKwData{Start}{start}
\SetKwData{End}{end}
\SetKwData{Cut}{cut}
\SetKwData{Mats}{mats}
\SetKwData{Primes}{primes}
\SetKwData{a}{a}
\SetKwData{b}{b}
\SetKwData{v}{v}
\SetKwData{cc}{c}
\SetKwFunction{Sorted}{Sorted}
\SetKwFunction{Set}{Set}
\SetKwFunction{FixBreak}{FixBreak}
\SetKwFunction{Denominator}{Denominator}
\SetKwFunction{IdMat}{IdentityMatrix}
\SetKwFunction{Matrices}{Matrices}
\SetKwFunction{RemTree}{RemTree}
\SetKwFunction{RationalShift}{RationalShift}
\SetKwFunction{Traces}{Traces}

\Fn{\Traces{$ \alpha, \beta, z, X $}}{
  \Input{$\alpha, \beta \in \bigr(\Q \cap [0, 1)\bigl)^r$, $z \in \Q$ and a bound $X$}
\Output{$\Hvalue{p}{\alpha}{\beta}{z} \pmod{p}$ for all good $p \le X$}
\If{$0 \in \alpha$}{
    $\alpha, \beta, z \colonequals \beta, \alpha, 1/z$\;
}
\Gammas $\colonequals$ \Sorted{ \Set{ $\alpha \cup \beta \cup \{0, 1\}$ } }\;
\For{good primes $p \le X$}{
    $\Result[p] \colonequals \IdMat{$2$}$\;
}
\For{\Start, \End consecutive elements of \Gammas}{
    $\b \colonequals \Denominator{\Start}$\;
    \For{$\cc \in (\Z /\mbox{\b}\Z)^\times$}{
      $\delta, \epsilon \colonequals $ \RationalShift{\Start, \cc} \tcp*[r]{Using Lemma~\ref{lem:shift}}
      \Mats $\colonequals$ \Matrices{$z$, \Start, $\delta$, $\epsilon$} \tcp*[r]{As in~\eqref{eq:Aic}}
        \Cut $\colonequals$ ($p \mapsto \lfloor \mbox{\End} \cdot (p-1) \rfloor - \lfloor \mbox{\Start} \cdot (p-1) \rfloor$)\;
        \Primes $\colonequals$ \{ good primes $p \equiv c \pmod{b}$, $p \le X$\}\;
        $\{C_i\} \colonequals$ \RemTreeWithSpacing{\Mats, \Primes, \Cut}\;
        \For{$i \colonequals 0, \ldots,  \#\Primes - 1$}{
          $p \colonequals \Primes[i]$\;
            $\Result[p] \colonequals \Result[p] \cdot \FixBreak{z, \Start, p}$ \tcp{As in~\eqref{eq:Tip}}
            $\Result[p] \colonequals \Result[p] \cdot C_i$\;
        }
    }
}
\For{good primes $p \le X$}{
    $\Result[p] \colonequals \Result[p]_{21} / \Result[p]_{11} \pmod{p}$\;
}
\Return{\Result}
}
\caption{Trace mod $p$}
\label{alg:traces}
\end{algorithm}

\begin{theorem} \label{T:main algorithm}
For fixed $\alpha, \beta$, Algorithm~\ref{alg:traces} is correct and runs in time
\[
O(X\log(X)^3).
\]
\end{theorem}

\begin{proof}
Correctness is immediate from Proposition~\ref{prop:mod p calculation}.
The runtime is dominated by the calls to Algorithm~\ref{alg:acc_rem_tree_with_spacing};
these calls takes place inside a loop over consecutive elements of $\alpha \cup \beta \cup \{0,1\}$
and a second loop over residue classes modulo a divisor of $b$.
These two loops together have length $O(rb)$;
combining with the runtime estimate from Theorem~\ref{thm:acc_rem_tree analysis}
(taking $B = b = O(X), H = O(\log X)$)
yields the desired result.
\end{proof}

\subsection{Implementation notes}

We have implemented Algorithm~\ref{alg:traces} in SageMath,
using a variant of Algorithm~\ref{alg:acc_rem_tree_with_spacing} implemented in C by Drew Sutherland
(see Remark~\ref{rmk:remainder forest}).
Our implementation is available at
\begin{center}
\url{https://github.com/edgarcosta/amortizedHGM},
\end{center}
and vastly outperforms SageMath and Magma while giving matching answers; see~Table~\ref{tab:timings} for sample timings.

\begin{table}[!ht]
\addtocounter{equation}{1}
\begin{tabular}{rlll|rlll|rl}
  $X$ & Alg.~\ref{alg:traces} & Sage & Magma
      &
  $X$ & Alg.~\ref{alg:traces}\\
\hline
 & & & &\\[-0.14in]
  $2^{10}$ & \texttt{0.07s} & \texttt{0.39s} & \texttt{0.11s}
           &
  $2^{18}$ & \texttt{1.81s} \\
         
  $2^{11}$ & \texttt{0.05s} & \texttt{0.68s} & \texttt{0.35s}
           &
  $2^{19}$ & \texttt{4.59s} \\

  $2^{12}$ & \texttt{0.06s} & \texttt{2.12s} & \texttt{1.29s}
           &
  $2^{20}$ & \texttt{10.71s} \\

  $2^{13}$ & \texttt{0.08s} & \texttt{7.39s} & \texttt{4.83s}
           &
  $2^{21}$ & \texttt{24.53s} \\
  
  $2^{14}$ & \texttt{0.12s} & \texttt{26.0s} & \texttt{18.24s}
           &
  $2^{22}$ & \texttt{58.0s} \\
  
  $2^{15}$ & \texttt{0.18s} & \texttt{92.27s} & \texttt{68.35s}
           &
  $2^{23}$ & \texttt{135s} \\

  $2^{16}$ & \texttt{0.34s} & \texttt{343s} & \texttt{280s}
           &
  $2^{24}$ & \texttt{322s} \\
           
  $2^{17}$ & \texttt{0.80s} & \texttt{1328s} & \texttt{1190s}
           &
  $2^{25}$ & \texttt{857s} \\
\end{tabular}
\caption{Comparison of Algorithm~\ref{alg:traces} against SageMath and Magma for $\alpha =  (\frac14, \frac12, \frac12, \frac34), \beta = (\frac13, \frac13, \frac23, \frac23)$ and $z = 1/5$. Note the observable difference between linear and quadratic complexity.}
\label{tab:timings}
\end{table}

\subsection{An example}

\begin{figure}
\includegraphics[width=0.75\textwidth]{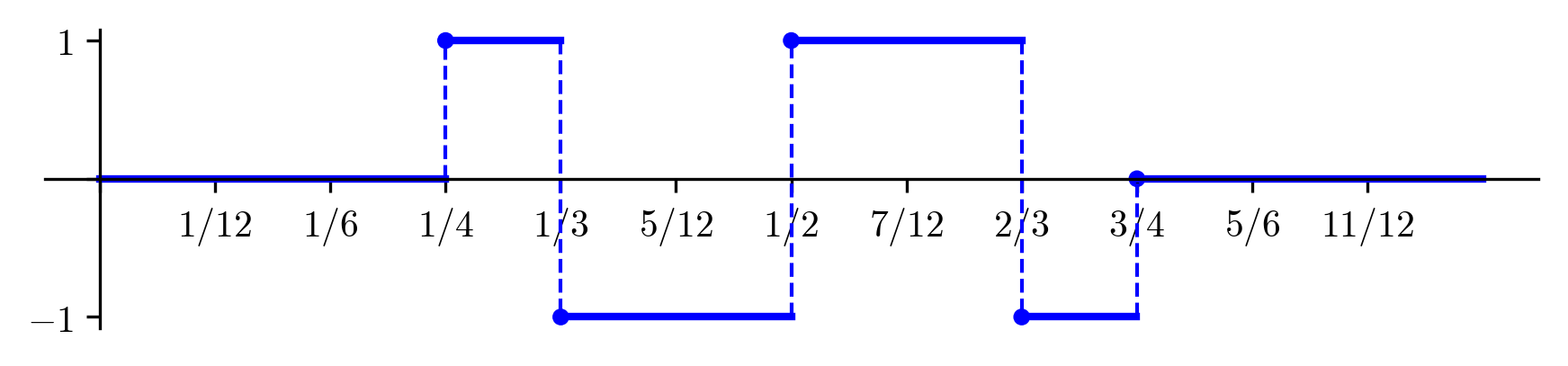}
\caption{$Z_{\alpha, \beta}(x)$ for $\alpha =  (\frac14, \frac12, \frac12, \frac34), \beta = (\frac13, \frac13, \frac23, \frac23)$\label{plot:zigzag}}
\end{figure}

Let $\alpha = (\frac14, \frac12, \frac12, \frac34), \beta = (\frac13, \frac13, \frac23, \frac23)$ and $z=\frac15$.  We plot the zigzag function in Figure~\ref{plot:zigzag}.  Using equation~\eqref{eq:mot_wt}, we see that $M^{\alpha, \beta}$ has weight $1$ and the intervals contributing to the computation of $\Hp$ are $(\gamma_2, \gamma_3) = (\frac13, \frac12)$ and $(\gamma_4, \gamma_5) = (\frac23, \frac34)$.
For the remainder of the example we will focus on the congruence class $p \equiv 7 \pmod{12}$.
Applying Lemma~\ref{lem:shift} to $\gamma_2 = \frac 13$ (resp. $\gamma_4 = \frac 23$), we obtain $\delta_2 = \frac23$ and $\epsilon_2=1$ (resp. $\delta_4 = \frac 13$ and $\epsilon_4 = 1$).
By~\eqref{eq:f and g} and~\eqref{eq:sigmai},
\begin{align*}
f_{2, 7}(k) &= 5184k^4 + 8640k^3 + 4428k^2 + 852k + 55, \\
g_{2, 7}(k) &= 25920k^4 + 69120k^3 + 63360k^2 + 23040k + 2880, \\
f_{4, 7}(k) &= 5184k^4 + 12096k^3 + 9612k^2 + 2820k + 175, \\
g_{4, 7}(k) &= 25920k^4 + 86400k^3 + 106560k^2 + 57600k + 11520,
\end{align*}
and $\sigma_2 = \sigma_4 = -1$.
Taking $p=67$, we obtain $(m_2, m_3) = (22, 33)$ and $(m_4, m_5) = (44, 49)$.
Using an accumulating remainder tree (or simple multiplication), we get
\begin{align*}
S_2(67) &= \left( \begin{matrix} 65 & 0 \\ 34 & 5 \end{matrix} \right), &
S_4(67) &= \left( \begin{matrix} 54 & 0 \\ 25 & 41 \end{matrix} \right).
\end{align*}
However, we can't ignore the other intervals: they may not contribute to the sum, but they do track the Pochhammer symbols.  Similar computations show
\[
S_0(67) = \left( \begin{matrix} 38 & 0 \\ 0 & 62 \end{matrix} \right), \,
S_1(67) = \left( \begin{matrix} 50 & 0 \\ 0 & 47 \end{matrix} \right), \,
S_3(67) = \left( \begin{matrix} 1 & 0 \\ 0 & 16 \end{matrix} \right), \,
S_5(67) = \left( \begin{matrix} 1 & 0 \\ 0 & 38 \end{matrix} \right).
\]
It remains to handle the break points. Using Definition~\ref{def:htau} we get
\begin{align*}
T_0(67) &= \left( \begin{matrix} 1 & 0 \\ 0 & 6 \end{matrix} \right), &
T_1(67) &= \left( \begin{matrix} 1 & 0 \\ 0 & 31 \end{matrix} \right), &
T_2(67) &= \left( \begin{matrix} 1 & 0 \\ -1 & 12 \end{matrix} \right), \\
T_3(67) &= \left( \begin{matrix} 1 & 0 \\ -1 & 40 \end{matrix} \right),&
T_4(67) &= \left( \begin{matrix} 1 & 0 \\ -1 & 40 \end{matrix} \right), &
T_5(67) &= \left( \begin{matrix} 1 & 0 \\ -1 & 31 \end{matrix} \right).
\end{align*}
Putting them all together, we get
\[
S(67) =  T_0(67) S_0(67) \cdots T_5(67) S_5(67)= \left( \begin{matrix} 21 & 0 \\ 33 & 21 \end{matrix} \right)
\]
yielding $\Hvalue{67}{\alpha}{\beta}{\frac15} \equiv \frac{33}{21} \equiv 59 \pmod{67}$.

% Sage code:
% T = [ matrix(2, elt) for elt in [[1, 0, 0, 6],
%                                 [1, 0, 0, 31],
%                                 [1, 0, -1, 12],
%                                 [1, 0, -1, 40],
%                                 [1, 0, -1, 40],
%                                 [1, 0, -1, 31]]]
% S = [ matrix(2, elt) for elt in [[38, 0, 0, 62],
%                                  [50, 0, 0, 47],
%                                  [65, 0, 34, 5],
%                                  [1, 0, 0, 16],
%                                  [54, 0, 25, 41],
%                                  [1, 0, 0, 38]]]
% res = 1
% for i in range(6):
%     res = res*T[i]*S[i]

\section{Future goals and challenges}

We would like to be able to compute
$\Hvalue{p^f}{\alpha}{\beta}{z} \pmod{p^e}$
in average polynomial time for general $e$ and $f$, but we currently only implement this for $e=f=1$.
We highlight the key points at which new ideas would be needed to achieve this goal.

\subsection{The case \texorpdfstring{$e>1$}{e > 1}}\label{subsection:e greater than 1}

Allowing $e>1$ creates two related issues where our computation exploits extra structure of the trace formula mod $p$: the replacement of $[z]$ with $z$, and the use of the functional equation in \eqref{use of functional equation} to compare two values of $\Gamma_p$ at arguments that differ by $\tfrac{1}{1-p}$.

Such issues can usually be resolved using the ``generic prime''
technique of~\cite[\S 4.4]{harvey-15}: make the average polynomial time computation carrying suitable nilpotent variables, then make a separate specialization for each $p$.

\subsection{The case \texorpdfstring{$f>1$}{f > 1}}

Allowing $f>1$ creates more serious issues because of the change in the definition of $\Gamma_q^*(x)$,
which interferes with our division of the summation into a fixed number of ranges.
To see this in more detail, fix $v \in \{0,\dots,f-1\}$.
For each $\gamma \in \alpha \cup \beta$, a break occurs when the value of $\{p^v(\gamma - \tfrac{m}{q-1})\}$
changes when $m$ goes to $m+1$; there are $p^v$ such breaks.

It is unclear whether one can rearrange the formula~\eqref{eq:Hq} to remedy this issue.
It may help to implement the method of Frobenius structures suggested
in~\cite{kedlaya-19}, which scales linearly in $p$ rather than $q$. 
We may then argue as in Theorem~\ref{T:L-series use case} to compute the first $X$ coefficients of an $L$-series in average polynomial time.

\bibliographystyle{alpha}
\bibliography{references}

\end{document}